\documentclass[english,12pt,oneside]{amsproc}
\usepackage[english]{babel}
\usepackage{a4wide}
\usepackage{amsthm}
\usepackage{graphics}
\usepackage{amsfonts, amssymb, amscd, amsmath}
\usepackage{latexsym}
\usepackage[matrix,arrow,curve]{xy}
\usepackage{mathabx}
\usepackage{color}
\usepackage{pbox}
\usepackage{tikz}
\usetikzlibrary{matrix,decorations.pathreplacing,positioning}
\usepackage{hyperref}
\usepackage{tikz-cd}
\usepackage[justification=centering]{caption}

\DeclareMathOperator{\Ker}{Ker}

\DeclareMathOperator{\Hom}{Hom} \DeclareMathOperator{\rk}{rk}

\newcommand{\Z}{\mathbb{Z}}
\newcommand{\R}{\mathbb{R}}

\newcommand{\RP}{\mathbb{R}P}

\newcommand{\dt}{\mathbb{Z}_2}

\newcounter{stmcounter}[section]

\newcounter{problcounter}

\numberwithin{equation}{section}

\theoremstyle{plain}
\newtheorem{cor}[stmcounter]{Corollary}

\newtheorem{thm}[stmcounter]{Theorem}
\newtheorem{thmNo}{Theorem}

\newtheorem{prop}[stmcounter]{Proposition}
\newtheorem{lem}[stmcounter]{Lemma}
\newtheorem{probl}[problcounter]{Problem}

\theoremstyle{definition}
\newtheorem{defin}[stmcounter]{Definition}

\theoremstyle{remark}
\newtheorem{ex}[stmcounter]{Example}
\newtheorem{rem}[stmcounter]{Remark}

\newtheorem{nonex}[stmcounter]{Non-Example}

\begin{document}
\title{Equivariantly formal $2$-torus actions of complexity one}

\author{Vladimir Gorchakov}
\address{Faculty of computer science, Higher School of Economics}
\email{vyugorchakov@gmail.com}

\date{\today}
\thanks{The article was prepared within the framework of the HSE University Basic Research Program}

\subjclass[2020]{Primary 57S12, 57S17, 57S25,  57N65, 57R91; Secondary 57R18, 55M35, 55R20}

\begin{abstract}
In this paper we study a specific class of actions of a $2$-torus $\Z_2^k$ on manifolds, namely, the actions of complexity one in general position. We describe the orbit space of equivariantly formal $2$-torus actions of complexity one in general position and restricted complexity one actions in the case of small covers. It is observed that the orbit spaces of such actions are topological manifolds. If the action is equivariantly formal, we prove that the orbit space is a $\Z_2$-homology sphere.  We study a particular subclass of these $2$-torus actions: restrictions of small covers to a subgroup of index 2 in general position. The subgroup of this form exists if and only if the small cover is orientable, and in this case we prove that the orbit space of a restricted 2-torus action is homeomorphic to a sphere. 
\end{abstract}

\maketitle

\section{Introduction}
Let a compact torus $T^k = (S^1)^k$ effectively act on a closed manifold $X^{2n}$ with nonempty finite fixed points set.  The number $n-k$ is called \textit{the complexity of the action}. The study of orbit spaces of complexity zero is a well-know subject of toric topology (See \cite{BP}). In  \cite{AyzCompl}, A.\,Ayzenberg showed that the orbit space of a complexity one action in general position is a topological manifold. In \cite{AyzMasEquiv}, A.\,Ayzenberg and M.\,Masuda described the orbit space of equivariantly formal torus actions of complexity one in general position. In \cite{AyzCher}, A.\,Ayzenberg and V.\,Cherepanov described torus actions of complexity one in non-general position. 
 
Similarly, we can study the orbit space of $\Z_2^k$ action on a manifold $X^n$. The group $\dt^k$ is a real analog of $T^k$, it is called a \textit{$2$-torus}. Similarly to the torus case the number $n-k$ is called \textit{the complexity of the action}. In \cite{YuEqForm}, L.\,Yu studied the orbit space of an equivariantly formal $2$-torus  action of complexity zero. In \cite{BT}, V.\,Buchstaber and S.\,Terzic proved that $Gr_{4,2}(\R)/\dt^3 \cong S^4$ and  $Fl_3(\R)/\dt^2 \cong S^3$ for real Grassmann manifold $G_{4,2}(\R)$ of $2$-planes in $\R^4$ and  the real manifold  $Fl_3(\R)$ of full flags in $R^3$. In \cite{GD}, D.\,Gugnin proved that $T^n/\dt^{n-1} \cong S^n$ for a certain action with isolated fixed points. These are examples of complexity one actions. 

The aim of this paper is to describe the orbit space of equivariantly formal $2$-torus  action of complexity one in general position. 
We also describe  restricted complexity one actions in the case of small covers. 

Let us give preliminary definitions and formulate the main results.

Let a  $2$-torus $\dt^{n-1}$ act effectively on a connected closed smooth manifold $X=X^n$ with nonempty set of fixed points. For a fixed point $x\in X^{\dt^{n-1}}$ of the action we have the tangent representation of $\dt^{n-1}$ at $x$.  Consider $\alpha_{x,1},\ldots,\alpha_{x,n}\in \Hom(\dt^{n-1},\dt)\cong \Z_2^{n-1}$, the weights of the tangent representation at $x$.
The action is said to be in \emph{general position} if  for any fixed point $x$, any $n-1$ of the weights $\alpha_{x,1},\dotsc,\alpha_{x,n}$ are linearly independent over $\dt$. Now we provide a coordinate description of an action in general position in the case of $\R^n$.

Let $G$ be a subgroup of $\mathbb{Z}^{n}_2$ consisting of elements of the following form: 
\begin{equation} 
   G = \{ (g_1, \dots, g_n) \in \dt^n : \Pi_{i=1}^n g_i = 1 \},
\end{equation} where $g_i \in\{-1,1\}$. 
Since $\dt^n$ acts coordinate-wise on $\R^n$, we have an induced action of $G$ on $\R^n$, which we call \textit{the standard complexity one action}.
\begin{rem}
    If an action of $\dt^{n-1}$ on $\R^n$ is in general position, then it is weakly equivalent to the standard complexity one action, see Proposition  \ref{StandardGenPos}.
\end{rem}

C.\,Lange and M.\,Mikhailova in \cite{Mikh}, \cite{La19} studied orbit spaces of general representations of finite groups. It follows from their results that $\R^n/G \cong \R^n$ for complexity one representation in general position. Using this result and some additional condition, which holds for equivariantly formal actions, we show that for a $\dt^{n-1}$-action on $X$ in general position the orbit space $Q = X/\dt^{n-1}$  is a topological manifold. If a $\dt^{n-1}$-action is in non-general position, then $Q$ is a topological manifold with boundary similar to the torus case studied in \cite{Cher} by V.\,Cherepanov.

The action of $\dt^{n-1}$ on $X$ is called \textit{equivariantly formal} if
\begin{equation}
    \dim_{\dt} H^*(X^{\dt^{n-1}};\dt) = \dim_{\dt} H^*(X;\dt),
\end{equation} 
    where $X^{\dt^{n-1}}$ is the fixed point set of $\dt^{n-1}$-action.
\begin{rem}
    If an $\dt^{n-1}$-action is in general position, then the fixed point set is finite and the action is equivariantly formal if and only if $|X^{\dt^{n-1}}| = \dim_{\dt} H^*(X;\dt)$.
\end{rem}

\begin{thmNo}\label{Th1}
Let $\dt^{n-1}$ act on $X = X^n$ equivariantly formal and in general position. Then the orbit space $Q = X/\dt^{n-1}$ is a topological manifold and $H^*(Q; \dt) \cong H^*(S^n; \dt)$.
\end{thmNo}
Let $X$ be a small cover and let $G \cong \dt^{n-1}$ be a subgroup of $\dt^n$ of index 2. Then we get the restricted $G$-action of complexity one on $X$. The subgroup $G$ is called a \textit{$2$-subtorus in general position}, if this $G$-action is in general position.
In the case of small covers, we notice that the condition of existence $2$-subtorus in general position is equivalent to the the condition of orientability of a small cover, which was proved by H.\,Nakayama and Y.\,Nishimura in \cite{SmallCovers}. 
\begin{thmNo}\label{Th2}
    Let $X$ be a small cover. There exists a $2$-subtorus in general position if and only if $X$ is orientable. If such $2$-subtorus exist, then it is unique. 
\end{thmNo}
If the $2$-subtorus in general position exists, then the orbit space is homeomorphic to the sphere. 
\begin{thmNo}\label{Th3}
   Let $X = X^n$ be an orientable small cover and let $G$ be the $2$-subtorus in general position. Then the orbit space $X/G$ is homeomorphic to the $n$-dimensional sphere~$S^n$.
\end{thmNo}
We have the following coordinate description of the  $2$-subtorus in general position.
\begin{rem}
        Let $X = X^n$ be an orientable small cover over a simple polytope $P$, let $p = F_{i_1} \cap \dots \cap F_{i_n}$ be a vertex of $P$ and $\lambda_{i_1}, \dots, \lambda_{i_n}$ be  the corresponding characteristic vectors. Taking the characteristic vectors as standard generators of $\dt^n$, consider the following subgroup:
   $$G = \{ (g_1, \dots, g_n) \in \dt^n : \Pi_{i=1}^n g_i = 1 \}.$$ 
 Then $G$ is the $2$-subtorus in general position.  
\end{rem}
Now we provide some examples of complexity one actions. In all examples below actions are equivariantly formal and in general position.  In examples $1.5$, $1.6$ we get restricted actions on small covers. 
\begin{ex}
Let $\dt$ act on $S^2$ by rotation on angle $180^\circ$ around an axis. Then $S^2/\dt \cong S^2$. 
\end{ex}
\begin{figure}[h]
	\centering
	\scalebox{0.2}{\includegraphics{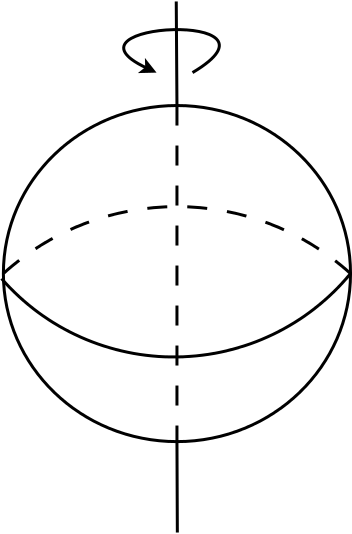}}
	\caption*{The action of $\dt$ on $S^2$}
\end{figure}

\begin{ex}
More generally, let $\dt$ act on a closed orientable surface $M_g$ of genus $g$ by rotation on angle $180^\circ$ around the axis.  Then $M_g/\dt \cong S^2$. If $g=1$, then it is the particular case of the next example. Notice that if $g=0$, then $M_0 = S^2$ is not a small cover. 
\end{ex}

\begin{figure}[h]
\begin{minipage}[h]{0.49\linewidth}
\center{\includegraphics[width=0.8\linewidth]{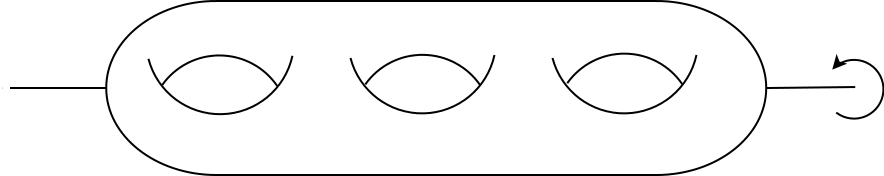}}
\caption*{The action of $\dt$ on $M_3$}
\end{minipage}
\hfill
\begin{minipage}[h]{0.49\linewidth}
\center{\includegraphics[width=0.55\linewidth]{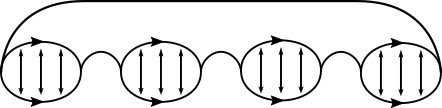} }
\caption*{The orbit space of the $\dt$-action on $M_3$}
\end{minipage}
\end{figure}

\begin{ex}
Consider a $\dt$-action on $S^1$ by the map $(x,y) \to (x,-y)$. Taking the n-fold product, we have a $\dt^n$-action on the $T^n$.  Let $$G = \{ (g_1, \dots, g_n) \in \dt^n : \Pi_{i=1}^n g_i = 1 \}$$
be the index 2 subgroup of orientation-preserving elements. In \cite{GD} it was proved that  $T^n/G \cong S^n$.
\end{ex}

Two examples below are not small covers. 

\begin{ex}
    Let $\dt^4$ act on $\R^4$ by the standard action. From this we get the effective action of $\dt^3$ on the real Grassmann manifold $G_{4,2}(\R)$ of $2$-planes in $\R^4$. In \cite{BT} it was proved that $G_{4,2}(\R)/\dt^3 \cong S^4$.
\end{ex}
\begin{ex}
    Let $\dt^3$ act on $\R^3$ by the standard action. From this we get the effective action of $\dt^2$ on the real full flag manifold $Fl_3(\R) $. In \cite{BT} it was proved that $Fl_3(\R)/\dt^2 \cong S^3$.
\end{ex}
Now we provide a possible connection of Theorem $3$ with the theory of \textit{n-valued groups}.  See \cite{ngroups} for the definition of $n$-valued group and other details. The following construction can be used to produce $n$-valued groups.

Let $G$ be a group, let $A$ be a finite group with $|A| = n$ and $\phi: A \to Aut(G)$ be homomorphism to the group of automorphisms of $G$. Then  we have an $n$-valued group structure on the orbit space $X = G/\phi(A)$, see \cite[Thm. 1]{ngroups}.

In example $1.6$ we get a $2^{n-1}$-valued topological group structure on the $n$-sphere $S^n$ for $n \geq 2$. The following problem was posed by V.\,M.\,Buchstaber.

\begin{probl}
Can Theorem $3$ be applied to other small covers to provide new examples of $2^{n-1}$-valued group structure on $S^n$?
\end{probl}
For this we need a small cover $X$ with the property that it is a group and $\dt^n$ acts by group automorphisms. 

\begin{nonex}
    We have that $\RP^3$ is diffeomorphic to $SO(3, \R)$, hence there is a Lie group structure on a small cover $\RP^3$. However, $\dt^3$ does not act by Lie group automorphisms. Indeed, since all automorphisms of $SO(3,\R)$ are inner, we have that $Aut(SO(3,\R)) \cong SO(3, \R)$. There is no finite subgroup in $SO(3,\R)$  isomorphic to $\dt^3$. Therefore, $\dt^3$ does not act by group automorphisms.  
\end{nonex}

\section{Preliminaries}
In this section we recall some general facts about group actions and $2$-torus actions on manifolds. The main reference about group actions is  \cite{Bredon}. The main reference about equivariantly formal $2$-torus actions is \cite{YuEqForm}.

Let a group $G$ act effectively on closed smooth manifold $X$. In this paper we consider only smooth action. For a point $x \in X$ let $Stab(x)$ denote the stabilizer subgroup of $G$ and $Gx$ the orbit of $x$.  We define the partition by orbit types 
\begin{equation}
    X = \bigsqcup_{H \subset G} X^{(H)}.
\end{equation}
Here $X^{(H)} = \{ x \in X : Stab(x) = H$\}. 

We denote the fixed point of subgroup $H$ by $X^H = \{x \in X: Stab(x) \subset H\}$.

Let $x \in X^{(H)}$ be a point with the  stabilizer subgroup~$H$. We  can define the \textit{tangent representation} of $H$ at $x$: 
$$
H \to GL(T_xX/T_xGx).
$$

Let $V_x$ denote $T_xX/T_xGx$. There is the following theorem about $G$-equivariant tubular neighborhood.
\begin{thm}[The Slice Theorem]
There exist a $G$-equivariant diffeomorphism from the $G \times_{Stab(x)} V_x$ onto a $G$-invariant neighborhood of the orbit $Gx$ in $X$, which send the zero section $G/Stab(x)$ onto the orbit $Gx$.
\end{thm}

We now recall the notation of an equivariantly formal action of 2-torus $\dt^k$, see \cite{YuEqForm} for the details. 

There is a classical result of E.\,Floyd. 
\begin{thm}[\cite{Floyd}] \label{thm:Hsiang}  
 For any paracompact $\dt^{k}$-space $X$ with finite cohomology dimension, the fixed point set
$X^{\dt^{k}}$ always satisfies 
\begin{equation} \label{eqSmithMain}
    \dim_{\Z_2} H^*(X^{\dt^{k}};\Z_2) \leq \dim_{\Z_2} H^*(X;\dt).
\end{equation}
 \end{thm}
 The next theorem says, when the equality in \eqref{eqSmithMain} holds.
 \begin{thm}[\cite{Hsiang75}]\label{EqFormalityTh}
     The equality in \eqref{eqSmithMain} holds if and only if the $E_2=E_{\infty}$ for the Serre spectral sequence of the fibration $X\rightarrow E \dt^{k}\times_{\dt^{k}} X \rightarrow B\dt^{k}$. 
 \end{thm}
\begin{defin}\label{EqFormality}
Let a $2$-torus $\dt^k$ act on a closed smooth manifold X. The action is called \emph{equivariantly formal} over $\dt$ if there is equality in \eqref{eqSmithMain}.
\end{defin}
Sometimes the degeneration of the Serre spectral sequence at $E_2$ is taken as the definition of equivariant formality. This definition is equivalent to Definition \ref{EqFormality} according to Theorem \ref{EqFormalityTh}. The notation of equivariant formality is similar to the corresponding notion in the theory of torus actions.  
\begin{rem}
If the fixed point set $X^{\dt^k}$ is finite, then $2$-torus action equivariant formal if and only if
\begin{equation}
|X^{\dt^k}| = \dim_{\dt} H^*(X;\dt).
\end{equation}

\end{rem}
This remark is useful to characterize equivariant formality of the particular $2$-torus actions. 

\begin{defin}
    Consider a non-free effective action of $\dt^n$ on a closed connected smooth manifold $X^n$. Manifold with such action is called  a \textit{$2$-torus manifold}.
\end{defin}

 L.\,Yu proved the following criteria of equivariant formality for $2$-torus manifolds in terms of its orbit space.
 \begin{thm}[\cite{YuEqForm}] \label{thm:ComplZero}
    Let $X$ be a $2$-torus manifold with orbit space $Q$.
      \begin{itemize}
      \item[(i)] $X$ is equivariantly formal if and only if $X$ is locally standard and $Q$ is mod $2$ face-acyclic.
      \item[(ii)] $X$ is equivariantly formal 
      and $H^*(X;\Z_2)$
     is generated by its degree-one part if and only if
      $X$ is locally standard and $Q$ is a mod $2$ homology polytope.
      \end{itemize}
   \end{thm}
\section{2-torus actions of complexity one in general position}
In this section we show that an orbit space of certain $2$-torus action of complexity one is a topological manifold. This section extends \cite{AyzCompl} to $2$-torus actions.

Let a  $2$-torus $\dt^{n-1}$ act effectively on a connected closed smooth manifold $X=X^n$ with nonempty set of fixed points. 

For a fixed point $x\in X^{\dt^{n-1}}$ of the action we the have tangent representation of $\dt^{n-1}$ at $x$.  Consider $\alpha_{x,1},\ldots,\alpha_{x,n}\in \Hom(\dt^{n-1},\dt)\cong \Z_2^{n-1}$, the weights of the tangent representation at $x$, i.e. 
$$
T_x X \cong V(\alpha_{x, 1}) \oplus \ldots \oplus V(\alpha_{x, 1})
$$
where $V(\alpha_{x, 1})$ is a 1-dimensional real representation given by $t \cdot y = \alpha_{x,1}(t) y $ for $t \in \dt^{n-1}$ and $y \in \mathbb{R}$.
\begin{defin} \label{GenPos}
The action is said to be in \emph{general position} if  for any fixed point $x$, any $n-1$ of the weights $\alpha_{x,1},\dotsc,\alpha_{x,n}$ are linearly independent over $\dt$.
\end{defin}
\begin{rem}
   From The Slice Theorem it follows that all weights of an action are non-zero if and only if the fixed point set is discrete. Hence, if $X$ is compact, then the fixed point set is finite. 
\end{rem}

Let an action be in general position. Since any $n-1$ of the weights  $\alpha_{x,1},\dotsc,\alpha_{x,n} \in \dt^{n-1}$ are linearly independent, we have $ \alpha_{x,1}+\dotsc+\alpha_{x,n} = 0$. Hence, for any $t \in \dt^{n-1}$ we have  
\begin{equation}\label{eqGenPos}
    \Pi_{i=1}^n \alpha(t)_{x,i} = 1.
\end{equation}

Moreover, the condition of general position implies that the tangent representation at any fixed point is faithful. This motivates the following construction.

Let $G$ be a subgroup of $\mathbb{Z}^{n}_2$ consisting of elements of the following form: 
\begin{equation} \label{GenPosCord}
   G = \{ (g_1, \dots, g_n) \in \dt^n : \Pi_{i=1}^n g_i = 1 \},
\end{equation} where $g_i \in\{-1,1\}$. Since $\dt^n$ acts coordinate-wise on $\R^n$, we have an induced action of $G$ on $\R^n$, which we call \textit{the standard complexity one action}. 

We show below that the orbit space of the standard complexity one action is homeomorphic to $\R^n$. Moreover, $G$ is the unique subgroup of $\dt^n$ of index 2 such that $\R^n/G \cong \R^n$.

Let $\chi: \dt^{n-1} \to GL_n(\R)$  be an action in general position with the weights $\alpha_1, \dots, \alpha_n$. Define $\phi: \dt^{n-1} \to G$ by the following formula $\phi(t) = (\alpha_1(t),\dots, \alpha_n(t))$. Since the action in general position is faithful representation and $\eqref{eqGenPos}$ holds, we have that $\phi$ is an isomorphism and the following diagram commutes:  
\begin{equation} \label{IsoGenPos}
    \begin{tikzcd}
\dt^{n-1} \arrow{r}{\chi} \arrow[swap]{d}{\phi} & GL_n(\R)\\
G  \arrow{ur}{\psi}
\end{tikzcd}
\end{equation}
Where $\chi$ is the action in general position and $\psi$ is the standard complexity one action of $G$ on $\R^n$ by coordinates, i.e. $g \cdot x = (g_1x_1, \dots, g_nx_n)$. Hence, we get the following:

\begin{prop} \label{StandardGenPos}
   Let $\dt^{n-1}$-action on $\R^n$ be in general position. Consider $G$ as in \eqref{GenPosCord}. Then this action is weakly equivalent to the standard complexity one action of $G$ on $\R^n$,  i.e. diagram  $\eqref{IsoGenPos}$ commutes.
\end{prop}

In the following proposition we prove that  all stabilizers of a $2$-torus action in general position are generated by \textit{rotations}, i.e. by orthogonal transformations whose fixed-point subspace has codimension two. 
\begin{prop} \label{StabOfG}
Consider  the standard complexity one action of $G$ on $\R^n$. Let $H = Stab(x)$ be the stabilizer subgroup of any $x \in \R^n$. Then the orbit space $\R^n/H$ is homeomorphic to $\R^n$. Moreover, $G$ is the only one subgroup of $\mathbb{Z}^{n}_2$ of index $2$ such that $\R^n/G \cong \R^n$.
\end{prop}
\begin{proof}
Let us describe the stabilizer $H$ of a point $x \in \R^n$. For $x = (x_1, \dots, x_n) \in \R^n$ let 
$$I = \{ i \in \{1, \dots, n\} : i \in I \text{ if } x_i =0\}=\{i_1, \dots, i_k\}$$ be the set of indices with zero coordinates of $x$. If $|I| < 2$, then the stabilizer subgroup is trivial, hence we can assume that $|I| \geq 2$.  Consider the following subgroup of $\dt^n$: 
\begin{equation}
    \dt^I = (\dt, 1)^I = \{(g_1, \dots, g_n) \in \dt^n : g_i \in \dt \text{ if } i \in I, \text{ otherwise } g_i = 1\}.
\end{equation}

We have that $H = \dt^I \cap G$. Taking $g_{i_1}, \dots, g_{i_k}$ as generators of $\dt^I$, we have that $H = \{(g_1, \dots, g_k) \in \dt^I : \prod^{k}_{1} g_i = 1 \}$. This means that $H$ generated by rotations. It follows from \cite[Thm. A]{La19} that the orbit space of this action is homeomorphic to $\R^n$. From this theorem also follows that any subgroup $H$ of $\dt^n$ such that $\R^n/H$ is homeomorphic to $\R^n$ must be generated by rotations, but $G = \{ (g_1, \dots, g_n) : \Pi_{i=1}^n g_i = 1 \}$ is the only subgroup of index $2$ generated by rotations. 
\end{proof}
\begin{cor} \label{StabOfGenPos}
Suppose  $\dt^{n-1}$ action on $\R^n$ is in general position.  Then $\R^n/Stab(x) \cong \R^n$ for the stabilizer subgroup $Stab(x)$ of any $x \in \R^n$. 
\end{cor}
\begin{proof}
  Since action is in general position, we have that the diagram \eqref{IsoGenPos} commutes. Therefore, the action is weakly equivalent to the action in Proposition \ref{StabOfG}. 
\end{proof}

\begin{rem}
From the previous corollary it follows that the stabilizer subgroup of an action in general position can not be arbitrary. For example, it can not be $H = \{(1,1,1,1), (-1,-1,-1,-1)\}$ and, indeed, $\R^4/H$ is homeomorphic to the open cone over $\RP^3$. This example shows importance of the condition that subgroup is generated by rotations. Description of all linear representations of  finite groups whose orbit spaces are homeomorphic to $\R^n$ can be found in \cite{La19}.
\end{rem}

For the global statement of the previous corollary we need the following condition on an action: 
\begin{equation}\label{eqCond}
\mbox{Every connected component of }X^H\mbox{ has a global fixed point,}
\end{equation}
where $X^H = \{ x \in X : H \subset Stab(x)$\}.

The following lemma shows that this condition holds for equivariantly formal action of $2$-torus.
\begin{lem}[{\cite[Lem. 3.2]{YuEqForm}}] \label{Lem:Induced-Formality}
   Suppose a $\Z_2^k$-action on a compact manifold $X$ is equivariantly formal. Then for every subgroup $H$ of $\Z_2^k$, the induced action of $\Z^k_2$( or $\dt^k/H$)  on every connected component $N$ of $X^H$ is equivariantly formal, hence $N$ has a $\Z^k_2$-fixed point.  
 \end{lem}
Now we can prove the following theorem about the orbit space of complexity one actions in general position. This is the first part of Theorem \ref{Th1} from the introduction.

\begin{thm}
Let $\dt^{n-1}$ act on a connected closed smooth manifold $X = X^{n}$. Suppose that the action is in general position and the condition \eqref{eqCond} holds. 
Then the orbit space $Q = X/\dt^{n-1}$ is a topological manifold.    
\end{thm}
\begin{proof}
   Let us denote $G = \dt^{n-1}$. Let $H = Stab(x)$ be the stabilizer subgroup of any $x$ in $X$, or equivalently  $x \in X^{(H)}$. Hence, $x \in N$, where $N$ is a connected component of $X^H$. It follows from the condition \eqref{eqCond} that $N$ has a global fixed point $x'$. By The Slice Theorem there exist a $G$ - equivariant neighborhood $U(x')$ of $x'$ such that $U(x')$ is $G$-diffeomorphic to $T_{x'}X$. Since the tangent representation of $H$ at $x$ depends only on a connected component of $X^H$, we can assume that $x$ is near $x'$, i.e. $x \in U(x')$. Therefore, $H$ is a stabilizer subgroup of an action $\dt^{n-1}$ on $\R^n$ in general position.  By The Slice Theorem every orbit has a $G$-equivariant neighborhood $U$ such that $U$ is $G$-equivariantly homeomorphic to $G \times_{H} V_x$, where $V_x = T_{x}X/T_{x}Gx$. Since the orbit $Gx$ is a discrete set, we have $V_x = T_{x}X$.  Therefore, $U/G \cong T_{x}X/H \cong \R^n$ by Corollary \ref{StabOfGenPos}. 
\end{proof}

\section{Equivariantly formal actions of complexity one}
In this section we show that the orbit space of an equivariantly formal action in general position is a $\dt$-homology sphere.  First of all, we introduce the notion of a face.

For an action of $\dt^{n-1}$ on $X$ consider the equivariant filtration
\[
\varnothing=X_{-1}\subset X_0 \subset X_1\subset\cdots\subset X_{n-1}=X,
\]
where $X_i$ is the union of orbits of size at most $2^i$ and $X_{-1} = \varnothing$. Notice that
\[
X_i = \{ x \in X : rk (Stab(x)) \geq n-1-i\} = \bigsqcup_{rk H \geq n-1 - i} X^{(H)}
\]
and each $X^{(H)}$ is the disjoint union of submanifolds of X.
There is an orbit type filtration of the orbit space $Q = X/\dt^{n-1}$:
\[
\varnothing=Q_{-1}\subset Q_0 \subset Q_1\subset\cdots\subset Q_{n-1}=Q, 
\]
where $Q_i = X_i/\dt^{n-1}$ and $Q_{-1} = \varnothing$.
\begin{defin}
A closure of any connected component of $Q_i\setminus Q_{i-1}$  is called \emph{a face of rank i}.  For a face $F$ of rank $i$ we define $F_{-1} = F \cap Q_{i-1}$.
\end{defin}
    Let $F$ be a face of $Q$, consider the quotient map $p: X \to Q$,  denote the preimage of $F$  by $X_F = p^{-1}(F)$. Let $G_F \subset \dt^{n-1}$ be the non-effective kernel of the $\dt^{n-1}$-action on $X_F$. Then $X_F$ is a connected component of $X^{G_F}$, therefore $X_F$ is a submanifold of $X$.
\begin{defin}
Let $F$ be a face of $Q$. The preimage $X_F = p^{-1}(F)$ of $F$ is called \emph{face submanifold} corresponding to $F$. For a face submanifold $X_F$ we define $X_{F_{-1}} = X_F \cap X_i$.
\end{defin}
In the next proposition we show that each $X_F$ is an equivariantly formal $2$-torus manifold.
\begin{prop}
Let $\dt^{n-1}$ act on $X$ equivariantly formal and in general position. Let $X_F$ be the face submanifold corresponding to a face $F$ of rank $i < n-1 $. Then $\dim X_F  = i $, hence $X_F$ is a equivariantly formal $2$-torus manifold with an action $\dt^{n-1}/G_F$.
\end{prop}
\begin{proof}
Let $G_F$ be the non-effective kernel of the $\dt^{n-1}$-action on $X_F$, we have that 
\begin{equation}
    |G_F| = 2^{n-1-i},
\end{equation}
since $\rk F = i$. Consider $p^{-1}(F^{\circ}) = \{ x \in X_F: Stab(x) = G_F\}$ is an open subset of $X_F$. Let $x' \in X_F$ be a global fixed point of action $\dt^{n-1}$ on $X$. Let $U(x')$ be $\dt^{n-1}$-equivariant neighborhood of $x'$ such that $U(x')$ is $\dt^{n-1}$-equivariant diffeomorphic to $T_x'X$. We have that $p^{-1}(F^\circ) \cap U(x') \cong V$, where $V = \{x \in T_{x'}X : Stab(x) = G_F\}$ is a linear subspace such that for any $x = (x_1, \ldots, x_n) \in V$ we have $x_i=0$, where $i \in I \subset \{1, \ldots, n\}$, i.e. some coordinates of $x$ are zero. Let $|I| = k$ be the number of zero coordinates. $G_F$~is generated by rotations, since it is the stabilizer subgroup of the standard complexity one action. Therefore,~$G_F$~contains all $(g_1,\ldots,g_n) \in \dt^{n}$ such that  $g_j = -1$ for all $j \in J$, where $J \subset I$ and $|J|$ is even.  Hence, we have
\begin{equation}
    |G_F| = \binom{k}{0} + \binom{k}{2} + \dots + \binom{k}{2\lfloor \frac{k}{2} \rfloor }  = 2^{k-1}.
\end{equation}

On the other hand, we have $|G_F| = 2^{n-1-i}$, therefore $k = n-i$. Hence, $\dim X_F = n - k = i $ and $\rk \dt^{n-1}/G_F = i$. This action is equivariantly formal by Lemma \ref{Lem:Induced-Formality}.
\end{proof}
In this paper all cohomology groups are taken with coefficients in $\dt$.

\begin{cor} \label{FaceCohomology}
For any face $F$ of rank $i < n-1$, we have $H^*(F,F_{-1}) =H^*(D^i, \partial D^i)$.
\end{cor}

\begin{proof} 
From Theorem \ref{thm:ComplZero} it follows that $F$ is mod 2 face-acyclic. Hence, $$H^*(F,F_{-1}) =H^*(D^i, \partial D^i)$$ by Lefschetz duality.
\end{proof}

Now we introduce  Atiyah–Bredon–Franz–Puppe sequence for equivariant cohomology  

\begin{multline}\label{eqABseqForX}
0\to H^*_{\dt^{k}}(X)\stackrel{i^*}{\to} H^*_{\dt^{k}}(X_0)\stackrel{\delta_0}{\to}
H^{*+1}_{\dt^{k}}(X_1,X_0)\stackrel{\delta_1}{\to}\cdots\\\cdots
\stackrel{\delta_{k-2}}{\to}H^{*+k-1}_{\dt^{k}}(X_{k-1},X_{k-2})\stackrel{\delta_{k-1}}{\to}H^{*+k}_{\dt^{k}}(X,X_{k-1})\to 0,
\end{multline}
where $\delta_i$ is the connecting homomorphism in the long exact sequence of equivariant cohomology of the triple $(X_{i+1},X_i,X_{i-1})$ and $X_i$ is the union of orbits of size at most $2^i$.
If an $\dt^{k}$-action on $X$ is  equivariantly formal, then this sequence is exact. For the proof see \cite{ABPF1},\cite{ABPF2}. Also see \cite{ABPFtorus} for the torus case. 
\begin{thm}[{\cite[Thm.\,10.2]{ABPF1}}]
    Suppose that an action of $\dt^{k}$ on $X$ is equivariantly formal. Then sequence \eqref{eqABseqForX} is exact. 
\end{thm}

From exactness of this sequence and previous corollary we immediately get the following result: 
\begin{prop} \label{RelHom}
    Let $\dt^{n-1}$ act on $X$ equivariantly formal and in general position. Then for the orbit space $Q$ we have $H^{i}(Q,Q_{n-2}) = 0$ for $i < n -1$.
\end{prop}
\begin{proof}
To be short denote $G = \dt^{n-1}$. Consider $i$-th term in \eqref{eqABseqForX} with $i \leq n - 2$:
$$
H^{*+i}_G(X_{i},X_{i-1})\cong \bigoplus_{F\colon\dim F=i}H^{*+i}_G(X_F,X_F\cap X_{i-1})\cong \bigoplus_{F\colon\dim F=i}H^{i}(F,F_{-1})\otimes H^*(BG_F).
$$

The first isomorphism follows from the equivariant version of Mayer-Vietoris sequence and the second isomorphism follows from the two facts: the action of group $G/G_F$ on $X_F \setminus (X_F \cap X_{i-1})$ is free and from Corollary \ref{FaceCohomology}. Therefore, $H^{*+i}_G(X_{i},X_{i-1}) =  0$ for $* < 0$ and $i \leq  n-2$. Consider $* < 0$. Then from exactness of sequence $\eqref{eqABseqForX}$ we get that $H_G^{i}(X,X_{n-2}) = 0$ for $i < n-1$. On the other hand, we have $H_G^{i}(X,X_{n-2}) \cong H^{i}(Q,Q_{n-2})$ since the action of $G$ on $X\setminus X_{n-2}$ is free. Hence, we have $H^{i}(Q,Q_{n-2}) = 0$ for $i < n-1$.
\end{proof}

Now we can prove Theorem \ref{Th1} from the introduction. 
\begin{thm}
Let $\dt^{n-1}$ act on $X = X^n$ equivariantly formal and in general position. Then the orbit space $Q = X/\dt^{n-1}$ is a $\dt$-homology $n$-sphere, i.e. $H^*(Q; \dt) \cong H^*(S^n; \dt)$.
\end{thm}
\begin{proof}
Consider the cohomology spectral sequence associated with the filtration 
\[
\varnothing=Q_{-1}\subset Q_0 \subset Q_1\subset\cdots\subset Q_{n-1}=Q, 
\]
where $Q_i = X_i/\dt^{n-1}$ and $Q_{-1} = \varnothing$:
\begin{equation} \label{ss}
    E_1^{p,q} = H^{p+q}(Q_p, Q_{p-1})\Rightarrow H^{p+q}(Q).
\end{equation}

We get that
\begin{equation}
    E_1^{p,q} = H^{p+q}(Q_p, Q_{p-1}) = \bigoplus_{F: rkF = p} H^{p+q}(F,F_{-1}).
\end{equation}

From the Corollary \ref{FaceCohomology} and Proposition \ref{RelHom} it follows that the only non-zero terms of the first page of the spectral sequence are the $0$-th row $(E_1^{p,0}, d^1)$ and the $n-1$-th column $(E_1^{n-1, q}, d^1)$. We have that the first page of the spectral sequence \eqref{ss} as shown on Figure~\ref{FirstPage}.

\begin{figure}[ht]
\centering
\begin{tikzpicture} [scale=0.7]
        \draw[->]  (0,0)--(7,0);
        \draw[->]  (0,-1)--(0,5);
        \draw (1,0)--(1,1); \draw (2,0)--(2,1); \draw (4,0)--(4,4); \draw (5,0)--(5,4);

        \draw (0,1)--(5,1); 

        \draw (0.5,0.5) node{$\ast$}; \draw (1.5,0.5) node{$\ast$}; \draw (3,0.5) node{$\cdots$};

        \draw (4.5,4.5) node{$\vdots$};
        
         \draw (4.5,2.5) node{$\ast$}; \draw (4.5,3.5) node{$\ast$}; 
         
        \draw (4.5,1.5) node{$\ast$}; 
        \draw (4.5,0.5) node{$\ast$};


        \draw (1,4) node{$E_1^{p,q}$};
        \draw (0.5,-0.3) node{\tiny $0$}; \draw (4.5,-0.3) node{\tiny $n-1$}; \draw (-0.3,0.5) node{\tiny $0$};
        \draw (6.5,-0.3) node{$p$}; \draw (-0.3,4.5) node{$q$};

        \draw (2,2) node{\large $0$}; \draw (3,-1) node{\large $0$}; 
        \draw (6,2) node{\large $0$};
\end{tikzpicture}
\caption{The first page of the spectral sequence \eqref{ss}}\label{FirstPage}
\end{figure}
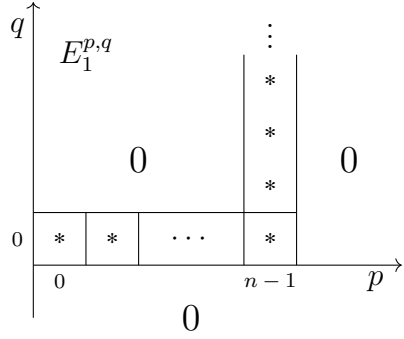

We claim that the differential complex in the $0$-th row $(E_1^{p,0}, d^1)$ is isomorphic to the degree $0$ part of the non-augmented version of Atiyah-Bredon-Franz-Puppe sequence \eqref{eqABseqForX}
$$
0\to  H^0_{\dt^{n-1}}(X_0)\stackrel{\delta_0}{\to}
H^{1}_{\dt^{n-1}}(X_1,X_0)\stackrel{\delta_1}{\to}\cdots
\stackrel{\delta_{n-3}}{\to}H^{n-2}_{\dt^{n-1}}(X_{n-2},X_{n-3})\stackrel{\delta_{n-2}}{\to}H^{n-1}_{\dt^{n-1}}(X,X_{n-2})\to 0.
$$
Indeed, we have the natural projection map $\pi: X_G \to X/G$ such that $\pi((X_{i})_{G}) \subset Q_i$ for any $i$. Therefore, $\pi$ induces a map between long exact sequences of $(X_{i+1},X_i,X_{i-1})_G$ and $(Q_{i+1},Q_i,Q_{i-1})$, hence $\pi^*$ commutes with $\delta_i$ for any $i$. We claim that the induced maps in cohomology 
$$\pi^* : H^i(Q_i,Q_{i-1}) \to H^{i}_{G}(X_i,X_{i-1})$$
are isomorphisms.

We have that 
$$H^{i}_G(X_{i},X_{i-1}) =  \bigoplus_{F\colon rkF=i}H^{i}_G(X_F,X_{F_{-1}}),$$ 
$$H^{i}(Q_i, Q_{i-1}) = \bigoplus_{F\colon rkF = i} H^{i}(F, F_{-1}).$$

Since $(X_F, X_{F_{-1}})$ fixed by $G_F$ and $EG = E(G_F) \times E(G/G_F)$, we have that $H_{G}^i(X_F, X_{F_{-1}}) = H^i_{G/G_{F}}(X_F, X_{F_{-1}}) \otimes_{\dt} H^0(BG_F)$. Therefore, it is enough to prove that the map $$\pi^* :H^{i}(F, F_{-1})  \to H^{i}_{G/G_F}(X_{F}, X_{F_{-1}})$$ is an isomorphism. We have that the action of $G/G_F$ is free on $X_F \setminus X_{F_{-1}}$, therefore $\pi^{-1}(x) = E(G/G_F)$ for any $x \in F \setminus F_{-1}$. Hence, from a relative version of the
Vietoris–Begle Theorem (see \cite{RelVerVB}), it follows that the induced map
$$\pi^* :H^{i}(F, F_{-1})  \to H^{i}_{G/G_F}(X_{F}, X_{F_{-1}})$$
is an isomorphism. Since the action is equivariantly formal, the Atiyah-Bredon-Franz-Puppe sequence \eqref{eqABseqForX} is exact. Therefore, the differential complex in the $0$-th row $(E_1^{p,0}, d^1)$ is  exact.

Hence, we have that the second page of the spectral sequence \eqref{ss} as shown on Figure~$\ref{SecondPage}$.

\begin{figure}[h]
\centering
\begin{tikzpicture} [scale=0.7]
        \draw[->]  (0,0)--(6,0);
        \draw[->]  (0,-1)--(0,4);
        \draw (1,0)--(1,1);  \draw (4,0)--(4,4); \draw (5,0)--(5,4);

        \draw (0,1)--(5,1); 

        \draw (0.5,0.5) node{$\dt$}; \draw (2.5,0.5) node{$0$}; 

        \draw (4.5,4.5) node{$\vdots$};
        
         \draw (4.5,2.5) node{$\ast$}; \draw (4.5,3.5) node{$\ast$}; 
         
        \draw (4.5,1.5) node{$\ast$}; 
        \draw (4.5,0.5) node{$0$};

        \draw (1,4) node{$E_2^{p,q}$};
        \draw (0.5,-0.3) node{\tiny $0$}; \draw (4.5,-0.3) node{\tiny $n-1$}; \draw (-0.3,0.5) node{\tiny $0$};
        \draw (6.5,-0.3) node{$p$}; \draw (-0.3,4.5) node{$q$};

        \draw (2,2) node{\large $0$}; \draw (3,-1) node{\large $0$}; 
        \draw (6,2) node{\large $0$};
\end{tikzpicture}

\caption{The second page of the spectral sequence \eqref{ss}} \label{SecondPage}
\end{figure}
We have that $E^{p,q}_{\infty} = 0$ for $0< p+q \leq n-1$, therefore $H^i(Q) = 0$ for $i \leq n-1$. On the other hand, $Q$ is a topological manifold, hence $H^i(Q) \cong H^i(S^n)$.
\end{proof} 

\section{Complexity one actions in case of small covers}
In this section we recall the definition of a small cover and prove  Theorem \ref{Th2}, \ref{Th3} from the introduction. For the details about small covers see \cite{DJ}. 
\begin{defin}
Let $P=P^n$ be a simple polytope of dimension $n$.
A small cover over $P$ is a smooth manifold  $X=X^n$ with a locally standard $\dt^n$-action such that the orbit space $X/\dt^n$ is diffeomorphic to a simple polytope $P$ as a manifold with corners. 
\end{defin}
 Let $\pi: X \to P$ be a small cover over $P$. For every face $F$ of $P$  and for every $x,y \in \pi^{-1}(F^{\circ})$ the stabilizer group of $x$ and $y$ is the same, i.e. $Stab(x) = Stab(y)$. Denote this stabilizer group by $G_F$. In particular, if $F$ is a facet, then $G_F$  is subgroup of rank one, hence $G_F = \langle \lambda(F)\rangle$ for some $\lambda(F) \in \dt^n$. Hence, we get  \textit{the characteristic function} 
 $$\lambda : \mathcal{F} \to \dt^n,$$
 from the set $\mathcal{F}$ of all facets of $P$.
 We denote $\lambda(F_i)$ by $\lambda_i$.
For a codimension k face $F$  we have that $F = F_1 \cap \dots \cap F_k$ for some facets $F_1, \dots, F_k \in \mathcal{F}$, then $G_F$ is a subgroup with rank equal to $k$ and generated by $\lambda_1, \dots, \lambda_k$. Therefore, we have the following $(*)$-condition
\\

$(*)$  Let $F = F_1 \cap \dots \cap  F_k$ be any codimension k face of $P$. Then $\lambda_1, \dots, \lambda_k$ are linearly independent in $\dt^n$.
\\

Conversely, a simple polytope $P$ and a map $\lambda: \mathcal{F} \to \dt^n$  satisfying the $(*)$-condition determine a small cover $X(P,\lambda)$ over $P$. For the construction of $X(P, \lambda)$ and other details see \cite{DJ}.
\begin{thm}[{\cite[Prop.\,1.8]{DJ}}]
    Let $X$ be a small cover over $P$ with characteristic function $\lambda: \mathcal{F} \to \dt^n$. Then $X$ and $X(P,\lambda)$ are equivariantly homeomorphic.
\end{thm}
Let $X$ be a small cover, let $G \cong \dt^{n-1}$ be a subgroup of $\dt^n$ of index 2, then we get an restricted $G$-action of complexity one on $X$.  The subgroup $G$ is called a \textit{$2$-subtorus in general position}, if this $G$-action is in general position.

\begin{rem}
    Notice that it is possible that a $2$-subtorus in general position does not exist.  For example,  for the small cover $\RP^2$ over $\Delta^2$ there is no  $2$-subtorus in general position. 
\end{rem}
However, we will see that if $X$ is an orientable small cover, then such $2$-subtorus exist. 
Every subgroup $G \subset \dt^n$ with rank $n-1$ is determined by some  non-zero linear functional $\xi \in (\dt^n)^*$ by the following $G = \Ker(\xi)$. We have the following criteria when $G$ is in general position. 

\begin{prop} \label{exsubtorus}
Let $X = X^n$ be a small cover, let $\lambda_i = \lambda(F_i)$ be characteristic vectors.  
    The $2$-subtorus $G = \Ker (\xi: \dt^n \to \dt)$ is in general position if and only if $\xi(\lambda_i) = 1$, in other words $\lambda_i \notin G$, for every facet $F_i \in \mathcal{F}$.
\end{prop}
\begin{proof}
    If $p = F_{i_1} \cap \dots \cap F_{i_n}$ is a vertex of $P$, then the elements of the dual basis $\lambda^*_{i_1}, \dots, \lambda_{i_n}^*$ are the weights of the tangent representation of $\dt^n$ in the corresponding fixed point.  Therefore, $G = \Ker (\xi)$ is in general position at this fixed point if and only if any $n-1$ of the weights $\lambda^*_{i_1}, \dots, \lambda_{i_n}^*$ are linearly independent in $(\dt^n)^*/\langle\xi\rangle.$

    Since  $\lambda^*_{i_1}, \dots, \lambda_{i_n}^*$ is a basis, we have that $\xi = a_1 \lambda^*_{i_1} + \dots + a_n \lambda^*_{i_n}$ in $\dt^n$ for $a_i \in \dt$. Therefore, we have  $a_1 \lambda^*_{i_1} + \dots + a_n \lambda^*_{i_n} = 0$ in $(\dt^n)^*/\langle\xi\rangle$ for $a_i \in \dt$. We have that any $n-1$ of the weights $\lambda^*_{i_1}, \dots, \lambda_{i_n}^*$ are linearly independent in $(\dt^n)^*/\langle\xi\rangle$ if and only if all $a_i$ are non-zero, i.e. $a_i = 1$  for every $i$. The last statement is equivalent to that $\xi(\lambda_{i_k})=1$ for all $i_k$.
\end{proof}
\begin{cor}\label{StGenPosSC}
    If $G \subset \dt^n$ is in general position, then $Stab(x) \not\subset G$ for every $x \in X$.
\end{cor}
\begin{proof}
    Indeed, Stab(x) is generated by $\lambda_i$ for some $i$. 
\end{proof}
\begin{rem}\label{GenPosSC}
    From the proof, we see that there exist only one $2$-subtorus in general position. If we choose a vertex $p = F_{i_1} \cap \dots \cap F_{i_n}$, then corresponding characterestic vectors $\lambda_{i_1}, \dots, \lambda_{i_n}$ is a basis of $\dt^n$. Taking these vectors as generators of $\dt^n$, we get that
    $$G = \{ (g_1, \dots, g_n) \in \dt^n : \Pi_{i=1}^n g_i = 1 \}$$

    in the multiplicative notation.
\end{rem}
The condition from Proposition \ref{exsubtorus} is related to the following result of H.\,Nakayama and Y. Nishimura. 
\begin{thm}[{\cite[Thm.\,1.7]{SmallCovers}}]
    A small cover $X=X^n$ is orientable if and only if there exist $\xi \in (\dt^n)^*$ such that $\xi(\lambda_i) = 1$ for every face $F_i \in \mathcal{F}$.
\end{thm}
From this result we get Theorem \ref{Th2} from the introduction.
\begin{cor} \label{ExistSC}
    Let $X$ be a small cover. There exists a $2$-subtorus in general position if and only if $X$ is orientable.
\end{cor}

Now we can prove Theorem \ref{Th3} from the introduction. 
\begin{thm}
    Let $X = X^n$ be an orientable small cover over, let $G$ be the $2$-subtorus in general position. Then the orbit space $X/G$ is homeomorphic to the $n$-dimensional sphere~$S^n$.
\end{thm}
\begin{proof}

Let $Q=X/G$ be the orbit space of the $G$-action, let $P=X/\dt^n$ be the orbit space of the $\dt^n$-action. By the definition of a small cover, $P$ is a simple polytope and $\text{dim} P = n$.  
Notice that $P = Q/(\dt^n/G)$ and we have the quotient map
$$
p: Q \to Q/(\dt^n/G).
$$

If $x \in P$ is a free $\dt^n$-orbit, i.e. $x$ in the interior of $P$, then $p^{-1}(x) = \dt$. Otherwise, by Corollary \ref{StGenPosSC} there exist $t \in Stab(x)$ such that $t \notin G$. Therefore, if a $\dt^n$-stabilizer group of $x$ is non-trivial, i.e. $x$ in the boundary of $P$, then it is a fixed point of $\dt^n/G$-action. Hence, for $x \in \partial P$ we have that $p^{-1}(x)$ is a single point. Since $P$ is contractible, the map $p: Q \to Q/(\dt^n/G)$ admit a section over the interior of $Q/(\dt^n/G) = P$. Therefore, we have
$$
Q \cong P \times \dt/\!\!\sim
$$
where $(x,1) \sim (x,-1)$ if $x \in \partial P$. Since $P$ is homeomorphic to the $n$-disc $D^n$, we have that 
$$
Q \cong D^n \sqcup D^n /\!\!\sim \cong S^n,
$$
which proves the statement. 
\end{proof}
\section*{Acknowledgements}
The author would like to thanks his advisor A.\,A.\,Ayzenberg for an invitation to toric topology, his help and guidance during this work, T.\,E.\,Panov for  helpful discussions about small covers. The author would like to thanks Eva Vlasova for her support and for drawing pictures.   

\end{document}